\documentclass{mcom-l}

\newtheorem{theorem}{Theorem}[section]
\newtheorem{lemma}[theorem]{Lemma}
\newtheorem{proposition}[theorem]{Proposition}
\theoremstyle{definition}

\newtheorem{example}[theorem]{Example}

\theoremstyle{remark}
\newtheorem{remark}[theorem]{Remark}

\numberwithin{equation}{section}

\usepackage{eurosym}
\usepackage{amsfonts}
\usepackage{amsmath}
\usepackage{amssymb}
\usepackage{amsthm}
\usepackage{graphicx}
\usepackage{xifthen}
\usepackage{tikz}
\usepackage{color}
\usepackage{array}
\usepackage{subcaption}
\usepackage{epsfig}
\usepackage{pgfplots}
\usepackage{pgffor}
\usepackage{bm}

\providecommand{\U}[1]{\protect\rule{.1in}{.1in}}
\usetikzlibrary{shapes,snakes}
\usetikzlibrary{shapes.geometric}
\usetikzlibrary{calc,shadows}
\pgfplotsset{compat = newest}
\usetikzlibrary{shapes}
\usetikzlibrary{plotmarks}
\usetikzlibrary{calc}
\usetikzlibrary{3d}

\newcommand{\area}{A}
\newcommand{\vol}{V}

\newcommand{\p}{\mathbf{p}}
\newcommand{\x}{\mathbf{x}}
\newcommand{\T}{\mathbf{t}}

\newcommand{\DCT}{\Delta_{\mathrm{CT}}}
\newcommand{\SCT}[2]{\mathbb{S}^{#1}_{#2}\bigl( \DCT\bigr)}
\newcommand{\DPS}{\Delta_{\mathrm{PS}}}
\newcommand{\SPS}[2]{\mathbb{S}^{#1}_{#2}\bigl( \DPS\bigr)}

\newcommand{\PCT}[4]{\begin{tikzpicture}[baseline=(current bounding box.center), scale=0.215]
\draw (-2,0)--(2,0)--(0,2*1.73205)--cycle;
\draw (-2,0)--(0,2*1.73205/3);
\draw (2,0)--(0,2*1.73205/3);
\draw (0,2*1.73205)--(0,2*1.73205/3);
\draw[fill,black!70] (-2,0) circle(1.3cm);
\node at (-2,0) {{\scriptsize {\rm {\color{white} \tiny$ #1$}}}};
\draw[fill,black!70] (2,0) circle(1.3cm);
\node at (2,0) {{\scriptsize {\rm {\color{white}\tiny$ #2$}}}};
\draw[fill,black!70] (0,2*1.73205) circle(1.3cm);
\node at (0,2*1.73205) {{\scriptsize {\rm {\color{white} \tiny$ #3$}}}};
\draw[fill,black!70] (0,2*1.73205/3) circle(.8cm);
\node at (0,2*1.73205/3) {{\scriptsize {\rm {\color{white} \tiny $ #4$}}}};
\end{tikzpicture}}

\newcommand{\PPS}[6]{\begin{tikzpicture}[baseline=(current bounding box.center), scale=0.215]
\draw (-2,0)--(2,0)--(0,2*1.73205)--cycle;
\draw (-2,0)--(1,2*1.73205/2);
\draw (2,0)--(-1,2*1.73205/2);
\draw (0,2*1.73205)--(0,0);
\draw[fill,black!70] (-2,0) circle(1.15cm);
\node at (-2,0) {{\scriptsize {\rm {\color{white}\tiny $ #1$}}}};
\draw[fill,black!70] (2,0) circle(1.15cm);
\node at (2,0) {{\scriptsize {\rm {\color{white}\tiny $#2$}}}};
\draw[fill,black!70] (0,2*1.73205) circle(1.15cm);
\node at (0,2*1.73205) {{\scriptsize {\rm {\color{white} \tiny $#3$}}}};
\draw[fill,black!70] (0,0) circle(0.75cm);
\node at (0,0) {{\scriptsize {\rm {\color{white} \tiny $#4$}}}};
\draw[fill,black!70] (1,2*1.73205/2) circle(0.75cm);
\node at (1,2*1.73205/2) {{\scriptsize {\rm {\color{white} \tiny $#5$}}}};
\draw[fill,black!70] (-1,2*1.73205/2) circle(0.75cm);
\node at (-1,2*1.73205/2) {{\scriptsize {\rm {\color{white} \tiny $#6$}}}};
\end{tikzpicture}}

\newcommand{\qd}[2]{\mathcal{Q}_{#1, #2}}

\newcommand{\Hct}[1]{\mathcal{M}_{#1}^{\mathrm{CT}}}

\newcommand{\Bct}[1]{\mathcal{B}_{#1}^{\mathrm{CT}}}
\newcommand{\Bps}[1]{\mathcal{B}_{#1}^{\mathrm{PS}}}

\begin{document}

\title[Quadrature rules for splines of high smoothness on triangles]{Quadrature rules for splines of high smoothness on uniformly refined triangles}


\author[S. Eddargani]{Salah Eddargani}
\address{Department of Mathematics, University of Rome Tor Vergata, Italy}
\email{eddargani@mat.uniroma2.it}

\author[C. Manni]{Carla Manni}
\address{Department of Mathematics, University of Rome Tor Vergata, Italy}
\email{manni@mat.uniroma2.it}

\author[H. Speleers]{Hendrik Speleers}
\address{Department of Mathematics, University of Rome Tor Vergata, Italy}
\email{speleers@mat.uniroma2.it}

\thanks{The authors are members of the research group GNCS (Gruppo Nazionale per il Calcolo Scientifico) of INdAM (Istituto Nazionale di Alta Matematica).
This work has been supported by the MUR Excellence Department Project MatMod@TOV (CUP E83C23000330006) awarded to the Department of Mathematics of the University of Rome Tor Vergata,
by a Project of Relevant National Interest (PRIN) under the National Recovery and Resilience Plan (PNRR) funded by the European Union -- Next Generation EU (CUP E53D23017910001),
by the Italian Research Center in High Performance Computing, Big Data and Quantum Computing (CUP E83C22003230001),
and by a GNCS project (CUP E53C23001670001).}


\date{}


\begin{abstract}
In this paper, we identify families of quadrature rules that are exact for sufficiently smooth spline spaces on uniformly refined triangles in $\mathbb{R}^2$. Given any symmetric quadrature rule on a triangle $T$ that is exact for polynomials of a specific degree $d$, we investigate if it remains exact for sufficiently smooth splines of the same degree $d$ defined on the Clough--Tocher 3-split or the (uniform) Powell--Sabin 6-split of $T$. We show that this is always true for $C^{2r-1}$ splines having degree $d=3r$ on the former split or $d=2r$ on the latter split, for any positive integer $r$.
Our analysis is based on the representation of the considered spline spaces in terms of suitable simplex splines.
\end{abstract}
\maketitle

\section{Introduction}
Quadrature rules are an efficient tool for the numerical approximation of definite integrals. They are typically computed as a linear combination of function values at selected points, called quadrature nodes. Relevant examples are Gauss--Legendre quadrature rules, which are the standard choice in finite element methods \cite{H04,LGCL79}. Multivariate integration has received a considerable attention in the literature; see \cite{C03,LB77,RVNI07} and references therein.
In particular, quadrature rules that exactly integrate quadratic and cubic polynomials on $m$-simplices ($ m \geq 1 $) were provided in \cite{HM56}. Numerical integration over non-triangular regions was addressed in \cite{SN11}.

Elementwise quadrature rules facilitate numerical integration of (piecewise polynomial) functions in finite element analysis by applying a given rule (that is exact for polynomials of a certain degree) separately to each element. However, the widespread growth of isogeometric analysis \cite{IGA} raised the need for adequate quadrature rules tailored for smooth splines, i.e., piecewise polynomial functions with high smoothness. In fact, elementwise rules are far from optimal for smooth splines as they do not take full advantage of the inter-element smoothness. For instance, considering $m$-dimensional tensor-product spline spaces of degree $d$ and smoothness $\rho$,
the number of quadrature nodes needed for exact integration per mesh element is $\lceil \frac{d+1}{2} \rceil^{m}$
for the elementwise Gauss–Legendre rule, whereas an optimal rule needs asymptotically only $(\frac{d-\rho}{2})^m$ nodes \cite{HCD17}. For high-dimensional problems, such difference becomes relevant, especially for highly smooth spline spaces. This motivated the recent focus on optimal and near-optimal quadrature rules for tensor-product spline spaces \cite{BC16b,BPDC20,CST17}. We refer the reader to \cite{CLSM19} for a survey on quadrature in the context of isogeometric analysis based on tensor-product splines.
 
The tensor-product setting, however, suffers from two main drawbacks:
it fails in describing complex domain geometries and does not allow for adequate local refinement.
Splines on triangulations offer an interesting alternative to tensor-product splines, in particular within the isogeometric analysis framework. They excel in several aspects, including improved domain parameterizations and efficient local refinement strategies \cite{JQ14,HC15,HCP13}.

Quadrature rules for smooth splines on triangulations have received relatively little attention. Typically, a triangular quadrature rule is employed separately over each triangle of the triangulation, in an elementwise fashion. To our knowledge, the problem of finding ad-hoc quadrature rules for smooth splines on (special) triangulations was only addressed in the papers \cite{BK19,ELMS24,KB19}. In \cite{KB19}
the authors show that the quadrature rule with four nodes, designed in \cite{HM56} to exactly integrate the $10$-dimensional space of cubic polynomials on a triangle, remains exact on the $12$-dimensional space of $C^1$ cubic splines defined on the Clough--Tocher 3-split of the same triangle \cite{CT65}.
Thus, instead of using twelve nodes (four nodes within each micro-triangle) for exact integration according to the elementwise approach, exactness on the whole space can already be achieved with only four appropriate nodes.
In the same spirit, in \cite{BK19} the authors find that two 3-node quadrature rules, proposed in \cite{HM56} for exact integration of the $6$-dimensional space of quadratic polynomials on a triangle, remain exact on the $9$-dimensional space of $C^1$ quadratic splines defined on the (uniform) Powell--Sabin 6-split of the same triangle \cite{PS77}.
In \cite{ELMS24} a 4-node quadrature rule is presented, which exactly integrates the $12$-dimensional space of $C^1$ quadratic splines defined on the Powell--Sabin 12-split of a triangle \cite{PS77}.

Inspired by \cite{BK19,KB19}, in this paper we investigate quadrature rules beyond the elementwise approach for two classes of smooth spline spaces (parameterized by a positive integer $r$), defined on the Clough--Tocher 3-split and the (uniform) Powell--Sabin 6-split.
We show that any symmetric quadrature rule that exactly integrates polynomials of degree $3r$ also maintains exactness on the larger space of $C^{2r-1}$ splines of degree $3r$ on the Clough--Tocher 3-split; see Theorem~\ref{thm:CT-quad}. Similarly, for any symmetric quadrature rule that is exact for polynomials of degree $2r$, exactness remains for $C^{2r-1}$ splines of degree $2r$ on the (uniform) Powell--Sabin 6-split; see Theorem~\ref{thm:PS-quad}.
We note that $C^{2r-1}$ smoothness is the highest smoothness for spaces of the prescribed degrees that gives a non-trivial spline space.

Our extension of the results in \cite{BK19,KB19} is twofold. Firstly, we deal with spline spaces of arbitrarily large degree. Secondly, we prove that the ``preservation of exactness'' is not confined to a specific quadrature rule but applies to any symmetric quadrature rule that exactly integrates the polynomials contained in the considered spline spaces.
Thus, while \cite{KB19} proved the exactness of a specific 4-node quadrature rule (three nodes along the medians and one at the barycenter) for $C^1$ cubic splines on the Clough--Tocher 3-split, it actually holds that any symmetric quadrature rule remains exact for this spline space as long as it achieves exactness for cubic polynomials. For instance, a symmetric 6-node quadrature rule that exactly integrates quartic polynomials also preserves accuracy for $C^1$-cubic splines on the Clough--Tocher 3-split --- a finding not contained in \cite{KB19}. A similar extension carries over to the quadrature rule on the (uniform) Powell--Sabin 6-split considered in \cite{BK19}.

Finally, we remark that our results are obtained by using a completely different approach compared to \cite{BK19,KB19}. They are based on representing the considered spline spaces by means of an appropriate basis consisting of simplex splines \cite{M79} and exploiting the properties of such functions.

The rest of the paper is divided in four sections. Section~\ref{sec:prelim} collects some preliminary material: it presents the notion of simplex splines (the main tool for our analysis), introduces symmetric quadrature rules based on the concept of orbits, and ends with recalling the construction of the macro-element splines we focus on. Section~\ref{sec:CT-split} deals with the study of symmetric quadrature rules on the Clough--Tocher 3-split, while Section~\ref{sec:PS-split} addresses the Powell--Sabin 6-split case. Finally, Section~\ref{sec:conclusion} concludes the paper with a brief discussion on possible extensions to the trivariate setting.

\section{Preliminaries}\label{sec:prelim}
In this section, we review some material regarding simplex splines, symmetric quadrature rules based on the concept of orbits, as well as the notion of macro-element splines and the specific splits used in this paper.

\subsection{Simplex splines}
We denote by $\vol_{m}(\Omega)$ the $m$-dimensional volume of $\Omega\subset\mathbb{R}^m$, and for $m=2$, we also use the notation $\area(\Omega):=\vol_2 (\Omega)$.
For a positive integer $d$, let $\mathcal{P}:=\lbrace \p_1, \ldots , \p_{d+3} \rbrace $ be a sequence of possibly repeated points in $\mathbb{R}^2$ such that $\area( \langle\mathcal{P}\rangle ) >0$, where $ \langle \cdot \rangle$ denotes the convex hull.
Moreover, let $\sigma := \langle \widehat{\p}_1,\ldots , \widehat{\p}_{d+3} \rangle $ be a simplex in $\mathbb{R}^{d+2}$ such that $\vol_{d+2} ( \sigma) >0$ and the projection $ \pi : \mathbb{R}^{d+2} \rightarrow \mathbb{R}^{2}$ of $\sigma$ into $\mathbb{R}^{2}$ satisfies $\pi (\widehat{\p}_i)=\p_i$, $i=1,\ldots , d+3$.
The unit-integral bivariate simplex spline $S_\mathcal{P} $ is defined geometrically as
$$
S_\mathcal{P}: \mathbb{R}^2 \rightarrow \mathbb{R},\quad S_\mathcal{P} ( \mathbf{x} ) := \frac{\vol_{d} (\sigma \cap \pi^{-1} ( \mathbf{x}) )}{\vol_{d+2} ( \sigma )}.
$$
For properties of $S_\mathcal{P} $ and proofs, see \cite{M79}. Here, we mention:
\begin{itemize}
\item The simplex spline $S_\mathcal{P} $ is a non-negative piecewise polynomial of total degree $d$ and support $\langle \mathcal{P} \rangle$.
\item For $d=0$ we have
$$
S_\mathcal{P} ( \mathbf{x} )= \begin{cases}
\frac{1}{ \area( \langle \mathcal{P} \rangle )}, & \text{if $\mathbf{x} \in$ interior of $\langle \mathcal{P} \rangle$}, \\
0, & \text{if $\mathbf{x} \notin  \langle \mathcal{P} \rangle$},
\end{cases}
$$
and the value of $S_\mathcal{P}$ on the boundary of $ \langle \mathcal{P} \rangle $ has to be dealt with separately.
\item $S_\mathcal{P}$ is $C^{d+1-\mu}$ continuous across a knot line, where $\mu$ is the number of knots including multiplicity on that knot line (knot lines are the lines in the complete graph of $\mathcal{P}$). 
\end{itemize}
In this paper, we consider the normalized bivariate simplex spline $N_\mathcal{P} $ defined as
\begin{equation}
\label{eq:normalized-simplex}
N_\mathcal{P} := \frac{\area(T)}{\binom{d+2}{2}} S_\mathcal{P},
\end{equation}
where $T$ is a non-degenerate triangle in $\mathbb{R}^2$.
Then, for $d=0$ we have
$$
 N_\mathcal{P} ( \mathbf{x} )= \begin{cases}
 \frac{\area(T) }{ \area( \langle \mathcal{P} \rangle )}, & \text{if $\mathbf{x} \in$ interior of $\langle \mathcal{P} \rangle$}, \\
 0, & \text{if $\mathbf{x} \notin  \langle \mathcal{P} \rangle$}.
 \end{cases}
$$
Moreover, $N_\mathcal{P}$ enjoys the following properties.
\begin{itemize}
\item \textbf{Derivative formula.} For any $\mathbf{u} \in \mathbb{R}^2$, and any $\alpha_1,\ldots, \alpha_{d+3} \, \in \mathbb{R}$ such that $\mathbf{u} = \sum_{i=1}^{d+3} \alpha_i\, \p_i$, $\sum_{i=1}^{d+3} \alpha_i = 0$,
$$
\partial_{\mathbf{u}} N_\mathcal{P} = d \sum_{i=1}^{d+3} \alpha_i N_{  \mathcal{P} \backslash \p_i }.
$$
\item \textbf{Recurrence relation.} For any $\mathbf{x} \in \mathbb{R}^2$, and any $\beta_1,\ldots, \beta_{d+3} \, \in \mathbb{R}$ such that $\mathbf{x}=\sum_{i=1}^{d+3} \beta_i \, \p_i$, $\sum_{i=1}^{d+3} \beta_i = 1$,
$$
N_\mathcal{P} ( \mathbf{x}) = \sum_{i=1}^{d+3} \beta_i N_{ \mathcal{P} \backslash \p_i }( \mathbf{x}).
$$
\item \textbf{Knot insertion.} For any $\mathbf{y} \in \mathbb{R}^2$, and any $\gamma_1,\ldots, \gamma_{d+3} \, \in \mathbb{R}$ such that $\mathbf{y} = \sum_{i=1}^{d+3} \gamma_i\, \p_i$, $\sum_{i=1}^{d+3} \gamma_i = 1$,
$$
N_\mathcal{P} = \sum_{i=1}^{d+3} \gamma_i N_{ \mathcal{P} \bigcup \mathbf{y} \backslash \p_i }.
$$
\end{itemize}

\subsection{Symmetric quadrature rules for triangles}
Quadrature rules are an efficient tool for the numerical approximation of integrals. Here we are interested in integrals of the form
$$
\int_T f(\x) {\rm{d}}\x,
$$
where $f$ is a given function on the triangle $T$.
Let $n\in \mathbb{N},$ $n>0$ and a function space $\mathbb{S}$ be given. We denote by
$$
\qd{n}{\mathbb{S}} (f) := \area (T) \sum_{i=1}^n \omega_i f(\T_i )
$$
an $n$-node quadrature rule that is exact for any function in the space $\mathbb{S}$, i.e., such that
$$
\qd{n}{\mathbb{S}} (f) = \int_T f(\x) {\rm{d}}\x\ \text{ for all } f \in \mathbb{S}.
$$
The points $\T_i\in T$, $i=1,\dots, n$, are the nodes of the quadrature rule and $\omega_i$ are the corresponding weights.
Usually, the space $\mathbb{S}$ is chosen as the space of polynomials of a given degree.

A quadrature rule $\qd{n}{\mathbb{S}}$ is said to be symmetric if it maintains its properties under rotation around the barycenter and reflection with respect to the medians of the triangle.
More precisely, if a point defined by its barycentric coordinates $( \alpha_1, \alpha_2, \alpha_3 )$ with respect to $T$ is a node of $\qd{n}{\mathbb{S}}$, then all points resulting from all possible permutations of these barycentric coordinates are also nodes of $\qd{n}{\mathbb{S}}$, with the same weight (see \cite{FJZC20,P15,WX03}). Recall that barycentric coordinates satisfy $\alpha_1+ \alpha_2+ \alpha_3=1$.

The concept of symmetric quadrature rules gives rise to three types of nodes, depending on the number of possible combinations obtained by permuting the three values in $( \alpha_1, \alpha_2, \alpha_3 )$. The first case is when all three coordinates are equal, which means $\alpha_1 = \alpha_2 = \alpha_3 = \frac{1}{3}$. The second case is when two coordinates are equal, i.e., of the form $( \theta, \theta, 1-2\theta )$, $\theta\neq\frac{1}{3}$. In the third case, all the coordinates are different, i.e., $( \theta, \eta, 1-\theta-\eta )$, $\theta\neq\eta\neq1-\theta-\eta$. These three cases are referred to as type-$j$ orbit, with $j = 0, 1, 2$, respectively; see \cite{FJZC20} and references therein. This terminology arises from the fact that, in each case, the nodes are situated in the same orbit. Specifically, the type-0 orbit comprises a single point located at the barycenter. A type-1 orbit consists of three points, each positioned on a median, with their coordinates represented by the three unique permutations of $( \theta, \theta, 1-2\theta )$. A type-2 orbit encompasses six points, located outside the medians, with their coordinates forming the six unique permutations of $( \theta, \eta, 1 - \theta - \eta )$.
Figure~\ref{fig:orbits} illustrates examples of these three orbit types.

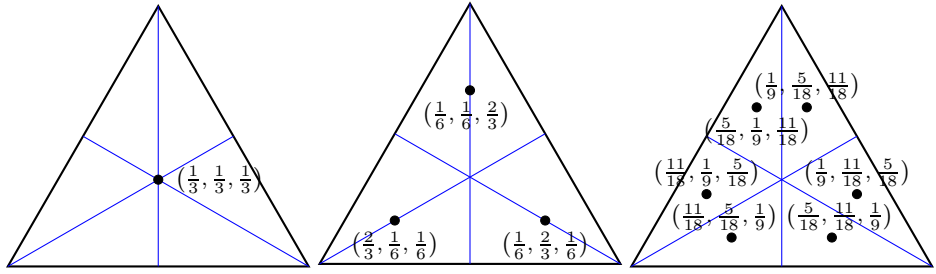
\begin{figure}[t!]
\centering
\begin{tikzpicture}[scale=2]
 \coordinate (A) at (-1,0); \coordinate (B) at (1,0); \coordinate (C) at (0,1.73205);
 \coordinate (R1) at ($(B)!0.5!(A)$); \coordinate (R2) at ($(B)!0.5!(C)$); \coordinate (R3) at ($(C)!0.5!(A)$);
 \draw[thick] (A) -- (B) -- (C) -- cycle; \draw[blue] (C) -- (R1); \draw[blue] (A) -- (R2); \draw[blue] (B) -- (R3);
  \pgfmathsetmacro{\a}{1/3}
  \pgfmathsetmacro{\aa}{(1-\a)/2}
 \coordinate (N1) at ($\a*(A)+ \aa*(B)+\aa*(C)$);
 \draw[fill] (N1) circle(0.03cm);
 \node[right=1mm] at (N1) {{\footnotesize $\left( \frac{1}{3}, \frac{1}{3}, \frac{1}{3}\right)$}};
\end{tikzpicture}
\begin{tikzpicture}[scale=2]
 \coordinate (A) at (-1,0); \coordinate (B) at (1,0); \coordinate (C) at (0,1.73205);
 \coordinate (R1) at ($(B)!0.5!(A)$); \coordinate (R2) at ($(B)!0.5!(C)$); \coordinate (R3) at ($(C)!0.5!(A)$);
 \draw[thick] (A) -- (B) -- (C) -- cycle; \draw[blue] (C) -- (R1); \draw[blue] (A) -- (R2); \draw[blue] (B) -- (R3);
  \pgfmathsetmacro{\a}{2/3}
  \pgfmathsetmacro{\aa}{(1-\a)/2}
 \coordinate (N1) at ($\a*(A)+ \aa*(B)+\aa*(C)$);
 \draw[fill] (N1) circle(0.03cm);
 \node[below=0.2mm] at (N1) {{\footnotesize $\left( \frac{2}{3}, \frac{1}{6}, \frac{1}{6}\right)$}};
  \coordinate (N2) at ($\a*(B)+ \aa*(C)+\aa*(A)$);
 \draw[fill] (N2) circle(0.03cm);
 \node[below=0.2mm] at (N2) {{\footnotesize $\left( \frac{1}{6}, \frac{2}{3}, \frac{1}{6}\right)$}};
 \coordinate (N3) at ($\a*(C)+ \aa*(A)+\aa*(B)$);
 \draw[fill] (N3) circle(0.03cm);
 \node[below=0.2mm] at (N3) {{\footnotesize $\left( \frac{1}{6}, \frac{1}{6}, \frac{2}{3}\right)\ $}};
\end{tikzpicture}
\begin{tikzpicture}[scale=2]
 \coordinate (A) at (-1,0); \coordinate (B) at (1,0); \coordinate (C) at (0,1.73205);
 \coordinate (R1) at ($(B)!0.5!(A)$); \coordinate (R2) at ($(B)!0.5!(C)$); \coordinate (R3) at ($(C)!0.5!(A)$);
 \draw[thick] (A) -- (B) -- (C) -- cycle; \draw[blue] (C) -- (R1); \draw[blue] (A) -- (R2); \draw[blue] (B) -- (R3);
  \pgfmathsetmacro{\a}{11/18}
  \pgfmathsetmacro{\b}{5/18}
   \pgfmathsetmacro{\c}{1/9}
 \coordinate (N1) at ($\a*(A)+ \b*(B)+\c*(C)$);
 \draw[fill] (N1) circle(0.03cm);
 \node[above=-0.3mm] at (N1) {{\footnotesize $\left( \frac{11}{18}, \frac{5}{18}, \frac{1}{9}\right)\ \ $}};
  \coordinate (N2) at at ($\a*(A)+ \c*(B)+\b*(C)$);
 \draw[fill] (N2) circle(0.03cm);
 \node[above=-0.2mm] at (N2) {{\footnotesize $\left( \frac{11}{18}, \frac{1}{9}, \frac{5}{18}\right)$}};
 \coordinate (N3) at ($\b*(A)+ \a*(B)+\c*(C)$);
 \draw[fill] (N3) circle(0.03cm);
 \node[above=-0.2mm] at (N3) {{\footnotesize $\ \ \left( \frac{5}{18}, \frac{11}{18}, \frac{1}{9}\right)$}};
  \coordinate (N4) at ($\b*(A)+ \c*(B)+\a*(C)$);
 \draw[fill] (N4) circle(0.03cm);
 \node[below=0.1mm] at (N4) {{\footnotesize $\left( \frac{5}{18}, \frac{1}{9}, \frac{11}{18}\right)$}};
  \coordinate (N5) at ($\c*(A)+ \a*(B)+\b*(C)$);
 \draw[fill] (N5) circle(0.03cm);
 \node[above=-0.2mm] at (N5) {{\footnotesize $\left( \frac{1}{9}, \frac{11}{18}, \frac{5}{18}\right)$}};
  \coordinate (N6) at ($\c*(A)+ \b*(B)+\a*(C)$);
 \draw[fill] (N6) circle(0.03cm);
 \node[above=-0.2mm] at (N6) {{\footnotesize $\left( \frac{1}{9}, \frac{5}{18}, \frac{11}{18}\right)$}};
\end{tikzpicture}
\caption{Schematic representation of type-0 (left), type-1 (center), and type-2 (right) orbits.}\label{fig:orbits}
\end{figure}

A symmetric quadrature rule that uses $n_0$ type-0 orbits $(n_0 \leq 1)$, $n_1$ type-1 orbits, and $n_2$ type-2 orbits is called a rule of type $[ n_0, n_1, n_2 ]$ and is denoted by $\qd{[ n_0, n_1, n_2 ]}{\mathbb{S}}$. The number of nodes for such a rule is
$$
n = n_0 + 3n_1 + 6n_2.
$$
Thus, the symmetric rule $\qd{[ n_0, n_1, n_2 ]}{\mathbb{S}}$ can be written as
\begin{align*}
\frac{\qd{[ n_0, n_1, n_2 ]}{\mathbb{S}} (f)}{\area (T)} &= \omega_0 f\biggl( \T\left( \frac{1}{3}, \frac{1}{3}, \frac{1}{3}\right)\biggr)
+ \sum_{i=1}^{n_1} \omega_{i,1} \sum_{\bm{\alpha} \in \Pi\lbrace \theta_{i,1}, \theta_{i,1},1-2 \theta_{i,1}\rbrace } f( \T(\bm{\alpha}) ) \\
&\quad\
+ \sum_{i=1}^{n_2} \omega_{i,2} \sum_{\bm{\alpha} \in \Pi\lbrace \theta_{i,2}, \eta_{i,2}, 1-\theta_{i,2} -\eta_{i,2} \rbrace } f( \T(\bm{\alpha}) ),
\end{align*}
for some
given
$
(\theta_{i,1},\theta_{i,1},1-2\theta_{i,1} )$, $ i=1,\ldots,n_1$, and $
(\theta_{i,2},\eta_{i,2}, 1-\theta_{i,2}-\eta_{i,2} )$, $ i=1,\dots,n_2$.
Here $\T (\bm{\alpha})$ denotes the point having the barycentric coordinates $ \bm{\alpha} := ( \alpha_1, \alpha_2, \alpha_3)$ and $\Pi \lbrace \beta_1, \beta_2, \beta_3\rbrace $ stands for all possible permutations of the triple $(\beta_1, \beta_2, \beta_3)$.
Figure \ref{fig:node-conf} shows examples of node configurations of symmetric quadrature rules that are exact for polynomials up to order $7$; see \cite{C72}.

\begin{figure}[t!]
\centering
\begin{tikzpicture}[scale=1.5]
\draw[thick] (-1,0)--(1,0)--(0,1.73205)--cycle;
\draw[blue] (-1,0)--(1/2,1.73205/2);
\draw[blue] (1,0)--(-1/2,1.73205/2);
\draw[blue] (0,0)--(0,1.73205);
\foreach \point  in {(-0.5, 0.288675), (0.5, 0.288675), (0., 1.1547)} {
 \draw[fill] \point  circle(0.03cm); }
\end{tikzpicture}
\quad
\begin{tikzpicture}[scale=1.5]
\draw[thick] (-1,0)--(1,0)--(0,1.73205)--cycle;
\draw[blue] (-1,0)--(1/2,1.73205/2);
\draw[blue] (1,0)--(-1/2,1.73205/2);
\draw[blue] (0,0)--(0,1.73205);
\foreach \point  in {(-0.4, 0.34641), (0.4, 0.34641), (0., 1.03923), (0., 0.57735)} {
 \draw[fill] \point  circle(0.03cm); }
\end{tikzpicture}
\quad
\begin{tikzpicture}[scale=1.5]
\draw[thick] (-1,0)--(1,0)--(0,1.73205)--cycle;
\draw[blue] (-1,0)--(1/2,1.73205/2);
\draw[blue] (1,0)--(-1/2,1.73205/2);
\draw[blue] (0,0)--(0,1.73205);
\foreach \point  in {(-0.725271, 0.158615), (0.725271, 0.158615), (0., 1.41482), (0.337845, 0.772405),  (-0.337845, 0.772405), (0., 0.18724)} {
 \draw[fill] \point  circle(0.03cm); }
\end{tikzpicture}
\\[0.5cm]
\begin{tikzpicture}[scale=1.5]
\draw[thick] (-1,0)--(1,0)--(0,1.73205)--cycle;
\draw[blue] (-1,0)--(1/2,1.73205/2);
\draw[blue] (1,0)--(-1/2,1.73205/2);
\draw[blue] (0,0)--(0,1.73205);
\foreach \point  in {(0., 0.57735),(-0.69614, 0.175433), (0.69614, 0.175433), (0., 1.38118), (0.410426, 0.81431), (-0.410426, 0.81431), (0., 0.103431)} {
 \draw[fill] \point  circle(0.03cm); }
\end{tikzpicture}
\quad
\begin{tikzpicture}[scale=1.5]
\draw[thick] (-1,0)--(1,0)--(0,1.73205)--cycle;
\draw[blue] (-1,0)--(1/2,1.73205/2);
\draw[blue] (1,0)--(-1/2,1.73205/2);
\draw[blue] (0,0)--(0,1.73205);
    \foreach \point  in {(-0.810733, 0.109273), (0.810733, 0.109273), (0., 1.5135), (-0.25214, 0.431777),  (0.25214, 0.431777), (0., 0.868496), (-0.32615, 0.0920499), (-0.583357, 0.537546), (0.32615, 0.0920499), (0.583357, 0.537546), (0.257207, 1.10245), (-0.257207, 1.10245)} {
        \draw[fill] \point  circle(0.03cm); }
\end{tikzpicture}
\quad
\begin{tikzpicture}[scale=1.5]
\draw[thick] (-1,0)--(1,0)--(0,1.73205)--cycle;
\draw[blue] (-1,0)--(1/2,1.73205/2);
\draw[blue] (1,0)--(-1/2,1.73205/2);
\draw[blue] (0,0)--(0,1.73205);
\foreach \point  in {(0., 0.57735),(-0.69614, 0.175433), (0.69614, 0.175433), (0., 1.38118), (-0.80461, 0.112809), (0.80461, 0.112809), (0., 1.50643), (-0.325579, 0.0843341), (-0.589754, 0.541899),  (0.325579, 0.0843341), (0.589754, 0.541899), (0.264175, 1.10582), (-0.264175, 1.10582)} {
 \draw[fill] \point  circle(0.03cm); }
\end{tikzpicture}
\caption{Node configurations of symmetric quadrature rules exact for  polynomials of degree $2$, $3$, and $4$ (top row, from left to right) and degree $5$, $6$, and $7$ (bottom row, from left to right).}\label{fig:node-conf}
\end{figure}
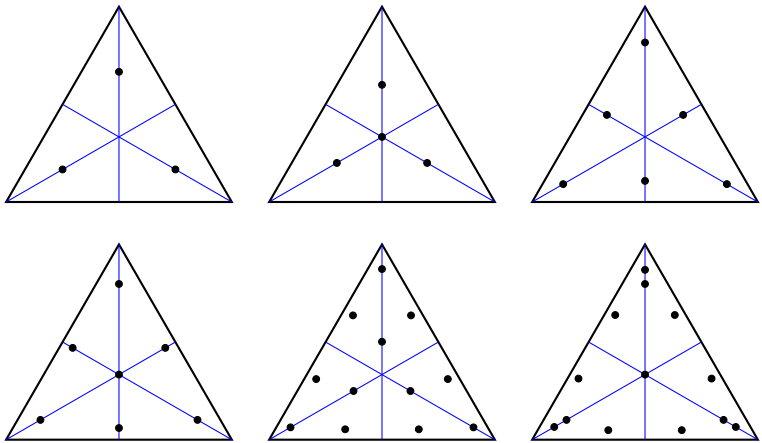

\subsection{Macro-element splines}
The spline space of global smoothness $\rho$ and degree $d$ on a triangulation $\Delta$ is defined by
$$
\mathbb{S}^\rho_d (\Delta) := \bigl\lbrace s \in C^{\rho} (\Delta): \, s_{\mid \tau}  \in \mathbb{P}_d  \text{ for all } \tau \in \Delta \bigr\rbrace ,
$$
where $\tau$ is any triangle in $\Delta$ and  $\mathbb{P}_d$ stands for the  $\binom{d+2}{2} $ dimensional linear space of bivariate polynomials of degree less than or equal to $d$.
Spline spaces with high smoothness compared to the degree are often preferred in applications but their use may be problematic because they may fail to have a stable dimension and may not allow for a local construction of the elements of the space; see \cite{LS07,MS24} and references therein for a detailed discussion.
These difficulties can be circumvented by using a proper refinement (or split) of each triangle of the given triangulation. This is called the macro-element  approach. Popular splits are the Clough--Tocher 3-split and the Powell--Sabin 6-split. In this paper, we focus on these splits in their symmetric (uniform) formulation as described in the following.

Let $T$ be a non-degenerate triangle with vertices $\p_1$, $\p_2$, and $\p_3$. If we connect the vertices of $T$ with the barycenter $\p_c := \frac{1}{3} \sum_{i=1}^3 \p_i $, we get three micro-triangles $\langle \p_1, \p_2, \p_c \rangle $, $\langle \p_2, \p_3, \p_c \rangle $, and $\langle \p_3, \p_1, \p_c \rangle $. The resulting triangulation of $T$ is known as the Clough--Tocher 3-split \cite{CT65} and denoted by $\DCT$; see Figure~\ref{fig:CTandPS} (left).
Furthermore, let $\p_{1,2}$, $\p_{2,3}$, and $\p_{3,1}$ be the midpoints of the edges $\langle \p_1, \p_2 \rangle $, $\langle \p_2, \p_3 \rangle $, and $\langle \p_3, \p_1 \rangle $, respectively. If we connect each midpoint by its opposite vertex of $T$, we get the six micro-triangles $\langle \p_1, \p_{1,2}, \p_c \rangle $, $\langle \p_{1,2}, \p_2, \p_c \rangle $, $\langle \p_2, \p_{2,3}, \p_c \rangle $, $\langle  \p_{2,3}, \p_3, \p_c \rangle $, $\langle \p_3, \p_{3,1}, \p_c \rangle $, and $\langle \p_{3,1}, \p_1,  \p_c \rangle $. The resulting triangulation of $T$ is the uniform case of the well-known Powell--Sabin 6-split \cite{PS77} and denoted by $\DPS$; see Figure~\ref{fig:CTandPS} (right).

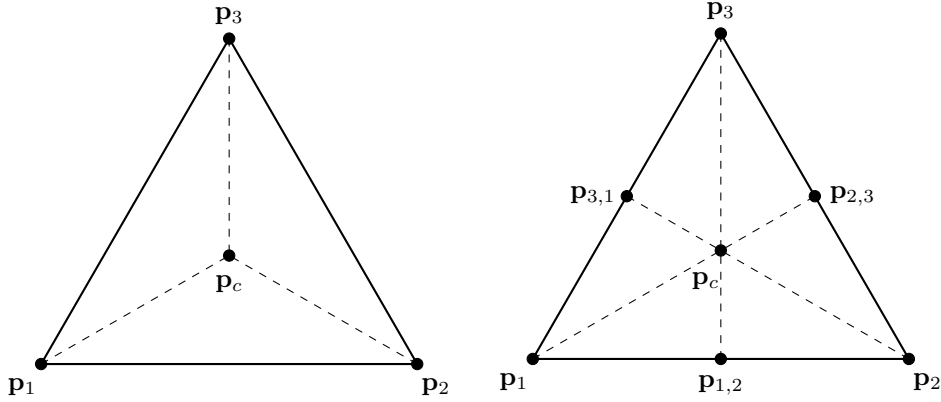
\begin{figure}[t!]
\centering
\begin{tikzpicture}[scale=2.5]
\draw[thick] (-1,0)--(1,0)--(0,1.73205)--cycle;
\draw[dashed] (-1,0)--(0,1.73205/3);
\draw[dashed] (1,0)--(0,1.73205/3);
\draw[dashed] (0,1.73205)--(0,1.73205/3);
\draw[fill] (-1,0) circle(0.03cm);
\node at (-1-0.1,0-0.12) {$\p_1$};

\draw[fill] (1,0) circle(0.03cm);
\node at (1+0.1,0-0.12) {$\p_2$};

\draw[fill] (0,1.73205) circle(0.03cm);
\node at (0,1.73205+0.12) {$\p_3$};

\draw[fill] (0,1.73205/3) circle(0.03cm);
\node at (0,1.73205/3-0.15) {$\p_c$};
\end{tikzpicture}
\hspace*{0.2cm}
\begin{tikzpicture}[scale=2.5]
\draw[thick] (-1,0)--(1,0)--(0,1.73205)--cycle;
\draw[dashed] (-1,0)--(1/2,1.73205/2);
\draw[dashed] (1,0)--(-1/2,1.73205/2);
\draw[dashed] (0,0)--(0,1.73205);
\draw[fill] (-1,0) circle(0.03cm);
\node at (-1-0.1,0-0.12) {$\p_1$};

\draw[fill] (1,0) circle(0.03cm);
\node at (1+0.1,0-0.12) {$\p_2$};

\draw[fill] (0,1.73205) circle(0.03cm);
\node at (0,1.73205+0.12) {$\p_3$};

\draw[fill] (-0.5,1.73205/2) circle(0.03cm);
\node at (-0.5-0.18,1.73205/2) {$\p_{3,1}$};

\draw[fill] (0.5,1.73205/2) circle(0.03cm);
\node at (0.5+0.2,1.73205/2) {$\p_{2,3}$};

\draw[fill] (0,0) circle(0.03cm);
\node at (0,0-0.14) {$\p_{1,2}$};

\draw[fill] (0,1.73205/3) circle(0.03cm);
\node at (-0.08,1.73205/3-0.16) {$\p_c$};
\end{tikzpicture}
\caption{The symmetric splits $\DCT$ (left) and $\DPS$ (right) of $T := \langle\p_1, \p_2, \p_3 \rangle$.}\label{fig:CTandPS}
\end{figure}

Our goal is to identify spline spaces of degree $d$ on either $\DCT$ or $\DPS$ that are exactly integrated by any given symmetric quadrature rule designed for the polynomial space $\mathbb{P}_d$, i.e., of type $\qd{[ n_0, n_1, n_2 ]}{\mathbb{P}_d}$. In other words, if we have a symmetric quadrature rule that exactly integrates $\mathbb{P}_d$, it will also exactly integrate the identified spline spaces of degree $d$ on either $\DCT$ or $\DPS$.
We begin this investigation by considering spline spaces defined on $\DCT$ (Section~\ref{sec:CT-split}). Afterwards, we address spline spaces defined on $\DPS$ (Section~\ref{sec:PS-split}).

\section{Clough--Tocher split}\label{sec:CT-split}
Let $\SCT{\rho}{d}$ be the spline space of degree $d$ and smoothness $\rho$ defined on the triangulation $\DCT$ of a given triangle $T$. The proposition below provides the dimension of this space.
\begin{proposition}
For any $0\leq\rho<d$, we have
\begin{equation}\label{eq:CT-dim-full}
\dim \, \SCT{\rho}{d} = \binom{\rho+2}{2} + 3 \binom{d-\rho+1}{2} + \sum_{j=1}^{d-\rho} \max \{ \rho+1-2j,0 \}.
\end{equation}
\end{proposition}
\begin{proof}
This follows immediately from \cite[Theorem~9.3]{LS07} by taking $n=3$ and $m_v=3$ in that theorem.
\end{proof}

As a natural generalization of the famous $C^1$ cubic Clough--Tocher macro-element \cite{CT65}, we consider piecewise polynomials of degree $d=3r$ on $\DCT$, for any integer $r\geq1$.
Formula \eqref{eq:CT-dim-full} shows that $\dim \, \SCT{2r}{3r} =\dim \, \mathbb{P}_{3r}$. This means that $\SCT{2r}{3r}$ only comprises polynomials of degree $3r$.
On the other hand, we have 
\begin{equation}\label{eq:CT-dim}
\dim \, \SCT{2r-1}{3r} =\dim \, \mathbb{P}_{3r} + 2.
\end{equation}
Therefore, the highest attainable smoothness for a non-trivial spline space of degree $d=3r$ on  $\DCT$ is $\rho=2r-1$.
In particular, we have
$$
\dim \, \SCT{1}{3} = 10 + 2,\quad  \dim \, \SCT{3}{6} = 28 + 2,\quad  \dim \, \SCT{5}{9} = 55 + 2.
$$
In the following, we verify whether symmetric quadrature rules that are exact on $\mathbb{P}_{3r}$ remain exact on $\SCT{2r-1}{3r}$.

As a first step, we construct a symmetric normalized simplex spline basis for $ \SCT{2r-1}{3r}$. In this perspective, we consider $\x \in T$, and
for integers $i, j, k, \ell$ we use the graphical notation
\begin{equation}\label{eq:CT-graphic}
\PCT{i}{j}{k}{\ell}
\end{equation}
to represent the normalized simplex spline $ N_\mathcal{P}$ in \eqref{eq:normalized-simplex} specified by the sequence of points
$$\mathcal{P}= \lbrace \p_1[i], \p_2[j], \p_3[k], \p_c[\ell] \rbrace, $$
where $\p[i]$ means that $\p$ has multiplicity $i$, i.e., it is repeated $i$ times. Some examples are visualized in Figure~\ref{fig:CT-splines}. Note that, when $\ell = 0$ and $i,j,k\geq 1$, $i+j+k = 3r+3$, the corresponding $N_\mathcal{P}$ is a Bernstein polynomial of degree $3r$ related to $T$. More precisely, for $i,j,k\geq 1$, $i+j+k = 3r+3$, we have
\begin{equation}\label{eq:Bernstein}
\PCT{i}{j}{k}{0} = \frac{(i+j+k-3)!}{(i-1)! (j-1)! (k-1)!} \alpha_1^{i-1} \alpha_2^{j-1} \alpha_3^{k-1},
\end{equation}
where $(\alpha_1, \alpha_2, \alpha_3)$ stand for the barycentric coordinates with respect to $T$.

\begin{figure}[t!]
\centering
{\includegraphics[scale=0.45]{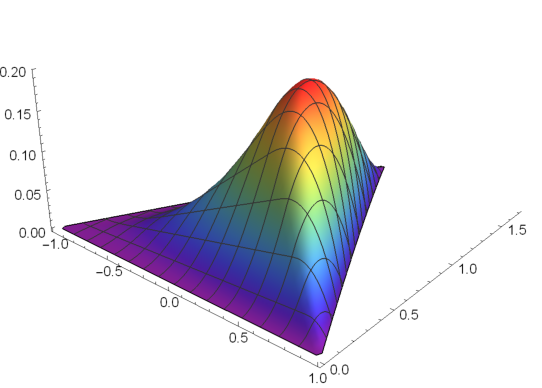}}
{\includegraphics[scale=0.45]{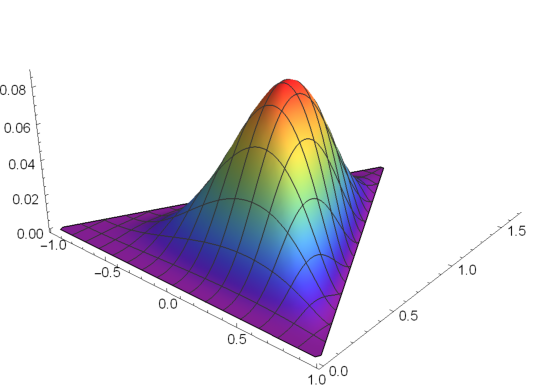}}
{\includegraphics[scale=0.45]{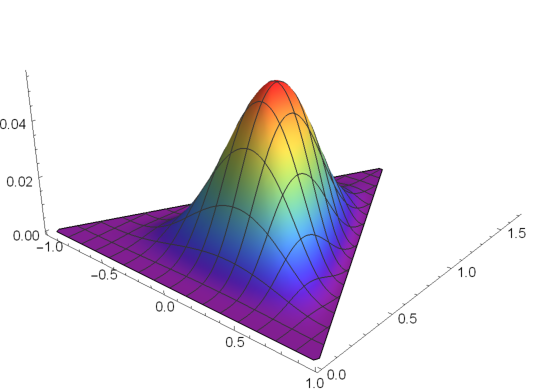}}
\\[0.2cm]
$\put(-135,20){ \PCT{1}{2}{2}{1}} \put(-10,20){ \PCT{2}{3}{3}{1}} \put(110,20){ \PCT{3}{4}{4}{1}}$
\caption{Plots of simplex splines as in \eqref{eq:CT-graphic} of degree $3$ (left), $6$ (center), and $9$ (right).}\label{fig:CT-splines}
\end{figure}
Inspired by the construction in \cite{LM18} for $r=1$ and looking at the dimension formula \eqref{eq:CT-dim}, to obtain a symmetric simplex spline basis for the space $\SCT{2r-1}{3r}$, one could replace the central Bernstein polynomial $ \PCT{r+1}{r+1}{r+1}{0}$ by three symmetric splines in $\SCT{2r-1}{3r}$. The knot insertion formula for inserting a knot at the barycenter gives
\begin{equation}\label{eq:Bern-central}
 \PCT{r+1}{r+1}{r+1}{0} =\frac{1}{3} \left( \PCT{r}{r+1}{r+1}{1} + \PCT{r+1}{r}{r+1}{1} +  \PCT{r+1}{r+1}{r}{1} \right) .
\end{equation}

The next lemma is useful to show the linear independence of the simplex splines we are dealing with.
\begin{lemma}\label{lem:CT-jump}
For $c_1,c_2,c_3\in\mathbb{R}$, let
\begin{equation}\label{eq:CT-polynomial}
p = c_1 \PCT{r}{r+1}{r+1}{1} + c_2 \PCT{r+1}{r}{r+1}{1} + c_3 \PCT{r+1}{r+1}{r}{1}.
\end{equation}
Then, $p \in \mathbb{P}_{3r}$ if and only if
$ c_1 = c_2 = c_3$.
\end{lemma}
\begin{proof}
If $c_1=c_2=c_3$ then from \eqref{eq:Bernstein} and \eqref{eq:Bern-central} it follows $p=3c_1\PCT{r+1}{r+1}{r+1}{0}\in \mathbb{P}_{3r}$.

Conversely, let us assume that $p\in \mathbb{P}_{3r}$.
Let ${\bf u}$ represent one of the micro-edges of $\DCT$ (the edges connecting the split point $\p_c$ with the vertices of $T$), and let $D_{{\bf v}}^{k,+} s$, $D_{{\bf v}}^{k,-} s$ denote the $k$-th derivative in the direction ${\bf v}$ of the spline $s$ restricted to the two micro-triangles of $\DCT$ sharing ${\bf u}$, considered in a given order.
Finally, let us define
$$
J_{{\bf u}, {\bf v}}^k s := D_{{\bf v}}^{k,+} s_{\mid {\bf u}} - D_{{\bf v}}^{k,-} s_{\mid {\bf u}}.
$$
Note that for a spline $s$ which is $C^k$ across ${\bf u}$, we have $J_{{\bf u}, {\bf v}}^k s = 0$ for any ${\bf v}$.

Let us fix $k=2r$ and
$ {\bf u}=\langle \p_1, \p_c \rangle $. Suppose that ${\bf n}$ is a direction normal to ${\bf u}$.
From the smoothness property of simplex splines we deduce
$$J_{{\bf u}, {\bf n}}^{2r} \PCT{r+1}{r}{r+1}{1} \neq 0, \quad J_{{\bf u}, {\bf n}}^{2r} \PCT{r+1}{r+1}{r}{1} \neq 0,$$
because those normalized simplex splines are only $C^{2r-1}$ across ${\bf u}$. On the other hand, both $p$ and  $\PCT{r}{r+1}{r+1}{1}$ are $C^{2 r}$  across ${\bf u} $. Consequently, from \eqref{eq:CT-polynomial} we obtain
\begin{align*}
0=J_{{\bf u}, {\bf n}}^{2 r} p
= c_2 J_{{\bf u},{\bf n}}^{2 r} \PCT{r+1}{r}{r+1}{1}+ c_3 J_{{\bf u},{\bf n}}^{2 r} \PCT{r+1}{r+1}{r}{1}.
\end{align*}
Due to the symmetry of the simplex splines $\PCT{r+1}{r}{r+1}{1}$ and $\PCT{r+1}{r+1}{r}{1}$, and the even order of the directional derivative (i.e., $2r$) we have
$$J_{{\bf u}, {\bf n}}^{2 r}  \PCT{r+1}{r}{r+1}{1} = - J_{{\bf u}, {\bf n}}^{2 r} \PCT{r+1}{r+1}{r}{1} \neq 0.$$
Thus, $c_2 = c_3$.
Similarly, when considering $ {\bf u}=\langle \p_2, \p_c \rangle $, we infer that $c_1 = c_3$. This concludes the proof.
\end{proof}

We now provide a normalized simplex spline basis for $\SCT{2r-1}{3r}$.

\begin{proposition}\label{prop:CT-basis}
The normalized simplex splines belonging to the set
\begin{equation}\label{eq:CT-basis}
\begin{aligned}
\Bct{r} &:= \left\lbrace \PCT{i}{j}{k}{0},\ i,j,k\geq 1,\ i+j+k=3r+3
\right\rbrace \setminus
\left\lbrace\PCT{r+1}{r+1}{r+1}{0} \right\rbrace \\
&\quad\ \bigcup \left\lbrace \PCT{r}{r+1}{r+1}{1},\, \PCT{r+1}{r}{r+1}{1},\, \PCT{r+1}{r+1}{r}{1} \right\rbrace
\end{aligned}
\end{equation}
form a basis for $\SCT{2r-1}{3r}$.
\end{proposition}
\begin{proof}
Firstly, it is directly seen that the cardinality of $\Bct{r}$ equals $\dim \, (\mathbb{P}_{3r})+2$, which agrees with the dimension of $\SCT{2r-1}{3r}$.
Then, we observe that all the elements of $\Bct{r}$ belong to $\SCT{2r-1}{3r}$.
Indeed, for $i,j,k\geq 1$, $i+j+k=3r+3$, we know that $\PCT{i}{j}{k}{0}$ is a Bernstein polynomial of degree $3r$, see \eqref{eq:Bernstein}, thus it belongs to $\SCT{2r-1}{3r}$. From the  properties of simplex splines, the remaining three elements in $\Bct{r}$ are piecewise polynomials of degree $3r$ and smoothness $C^{2r-1}$ across the micro-edges of $\DCT$, so they belong to $\SCT{2r-1}{3r}$ as well.

Finally, we show that the functions in \eqref{eq:CT-basis} are linearly independent. Let
$$
\Hct{r}:= \left\lbrace \left( i, j, k, \ell\right):\, \PCT{i}{j}{k}{\ell} \in   \Bct{r} \right\rbrace
$$
be the set of quadruples indicating the knot multiplicities of the simplex splines in \eqref{eq:CT-basis}. Suppose there exist real coefficients $ c_{i, j, k, \ell} $ such that
$$ 
\sum_{(i, j, k, \ell) \in \Hct{r}} c_{i, j, k, \ell} \PCT{i}{j}{k}{\ell}(\mathbf{x}) = 0, \quad \forall \ \mathbf{x} \in T.
$$
 Thus,
\begin{equation*}
\begin{aligned}
 &\sum_{\left( i, j, k, 0\right) \in \Hct{r}} c_{i, j, k, 0} \PCT{i}{j}{k}{0} (\x) \\
 & = -\left( c_{r, r+1, r+1, 1} \PCT{r}{r+1}{r+1}{1} + c_{r+1, r, r+1, 1} \PCT{r+1}{r}{r+1}{1} + c_{r+1, r+1, r, 1} \PCT{r+1}{r+1}{r}{1} \right) (\x).
\end{aligned}
\end{equation*}
The left-hand side of 
the above equation
is a polynomial of degree less than or equal to $3r$. Applying Lemma~\ref{lem:CT-jump} gives $ c_{r, r+1, r+1, 1} = c_{r+1, r, r+1, 1} = c_{r+1, r+1, r, 1} = c$. Then, from \eqref{eq:Bern-central} we deduce
$$
\sum_{(i, j, k, 0) \in \Hct{r}} c_{i, j, k, 0} \PCT{i}{j}{k}{0}(\x) = - 3 c \PCT{r+1}{r+1}{r+1}{0}(\x).
$$
By the linear independence of Bernstein polynomials of degree $3r$, this implies that $ c_{i, j, k, \ell} = 0 $ for all $ (i, j, k, \ell) \in \Hct{r} $.
This concludes the proof.
\end{proof}

We are now ready to prove the main result of this section.
\begin{theorem}\label{thm:CT-quad}
Any quadrature rule $\qd{[n_0, n_1, n_2]}{\mathbb{P}_{3r}}$ is exact on the spline space $\SCT{2r-1}{3r}.$
\end{theorem}
\begin{proof}
It suffices to show that the rule produces the exact value for the integral of all the elements of a basis of the space.
Since $\qd{[n_0, n_1, n_2]}{\mathbb{P}_{3r}}$ is exact on $\mathbb{P}_{3r}$, by Proposition~\ref{prop:CT-basis}, we only need to prove that $\qd{[n_0, n_1, n_2]}{\mathbb{P}_{3r}}$ is exact for $\PCT{r}{r+1}{r+1}{1}$, $\PCT{r+1}{r}{r+1}{1}$, $\PCT{r+1}{r+1}{r}{1}$.

Starting from \eqref{eq:Bern-central} and exploiting the symmetry of both the quadrature rule and the simplex splines, we obtain
\begin{align*}
&\qd{[n_0, n_1, n_2]}{\mathbb{P}_{3r}} \left( \PCT{r+1}{r+1}{r+1}{0} \right) = \qd{[n_0, n_1, n_2]}{\mathbb{P}_{3r}} \left( \PCT{r}{r+1}{r+1}{1} \right) \\
&\quad = \qd{[n_0, n_1, n_2]}{\mathbb{P}_{3r}} \left( \PCT{r+1}{r}{r+1}{1} \right) = \qd{[n_0, n_1, n_2]}{\mathbb{P}_{3r}} \left( \PCT{r+1}{r+1}{r}{1} \right).
\end{align*}
Furthermore, again from \eqref{eq:Bern-central} and the symmetry of the simplex splines, we get
$$
\int_T \PCT{r+1}{r+1}{r+1}{0}  = \int_T \PCT{r}{r+1}{r+1}{1} = \int_T  \PCT{r+1}{r}{r+1}{1}= \int_T \PCT{r+1}{r+1}{r}{1}.
$$
Since $\qd{[n_0, n_1, n_2]}{\mathbb{P}_{3r}}$ is exact for $\PCT{r+1}{r+1}{r+1}{0}$ , it follows that it is also exact for $\PCT{r}{r+1}{r+1}{1}$, $\PCT{r+1}{r}{r+1}{1}$, and $\PCT{r+1}{r+1}{r}{1}$. This concludes the proof.
\end{proof}

\begin{example}
The Hammer--Stroud quadrature rule
\begin{equation}\label{eq:HS-cubic}
\mathcal{Q}_{\mathbb{P}_3}^{\mathrm{HS}} (f) := \area (T) \left(  \frac{25}{48} \sum_{ \bm{\alpha} \in \Pi \left\lbrace \frac{3}{5}, \frac{1}{5}, \frac{1}{5} \right\rbrace } f \left( \T \left(\bm{\alpha} \right)\right) -\frac{9}{16} f \left( \T \left( \frac{1}{3}, \frac{1}{3}, \frac{1}{3}\right)\right) \right)
\end{equation}
is exact for cubic polynomials \cite{HM56}. The rule $\mathcal{Q}_{\mathbb{P}_3}^{\mathrm{HS}}$ is a quadrature rule of type $\qd{[1, 1, 0]}{\mathbb{P}_{3}}$, thus by Theorem~\ref{thm:CT-quad} (with $r=1$), it  exactly integrates the space $\SCT{1}{3}$. The same result has already been obtained in \cite{KB19}. However, the methodology employed in \cite{KB19} differs significantly from the approach used in this work.
\end{example}

\begin{remark}
The smoothness $C^{2r-1}$ is the minimum for splines of degree $3r$ on $\DCT$ to ensure that any quadrature rule of type $\qd{[n_0, n_1, n_2]}{\mathbb{P}_{3r}}$ remains exact on the spline space.
This can be seen by considering the simplest case $r=1$ and 
the normalized simplex spline $\PCT{3}{2}{0}{1}$, which is supported on $\langle\p_1, \p_2, \p_c \rangle $, $C^0$ across the micro-edge $\langle \p_1, \p_c \rangle $ and $C^1$ across the micro-edge $\langle \p_2, \p_c \rangle $. This non-trivial function belongs to $\SCT{0}{3}$ and is non-negative over the triangle $T$, so it has a positive integral over $T$. However, the Hammer--Stroud rule $\mathcal{Q}_{\mathbb{P}_3}^{\mathrm{HS}}$ defined in \eqref{eq:HS-cubic} yields a zero value when applied to $\PCT{3}{2}{0}{1}$ because the function vanishes along all the micro-edges of $\DCT$.
Therefore, the rule $\mathcal{Q}_{\mathbb{P}_3}^{\mathrm{HS}}$ fails to achieve exactness on the space $\SCT{0}{3}$.
\end{remark}

\section{Powell--Sabin split}\label{sec:PS-split}
Let $\SPS{\rho}{d}$ be the spline space of degree $d$ and smoothness $\rho$ defined on the triangulation $\DPS$ of a given triangle $T$. The next proposition provides the dimension of this space.
\begin{proposition}
For any $0\leq\rho<d$, we have
\begin{equation}\label{eq:PS-dim-full}
\dim \, \SPS{\rho}{d} = \binom{\rho+2}{2} + 6 \binom{d-\rho+1}{2}+\sum_{j=1}^{d-\rho} \max \{ \rho+1-2j,0 \}.
\end{equation}
\end{proposition}
\begin{proof}
This follows immediately from \cite[Theorem~9.3]{LS07} by taking $n=6$ and $m_v=3$ in that theorem.
\end{proof}

As a natural generalization of the famous $C^1$ quadratic Powell--Sabin macro-element \cite{PS77}, we consider non-trivial spline spaces of degree $d=2r$ and maximal smoothness $\rho=2r-1$ on $\DPS$, for any integer $r\geq1$.
This spline space is larger than $\mathbb{P}_{2r}$ because the dimension formula \eqref{eq:PS-dim-full} gives
$$
\dim \, \SPS{2r-1}{2r} = \dim \, \mathbb{P}_{2r} + 3.
$$
In particular, we have
$$
\dim \, \SPS{1}{2} = 6 + 3,\quad  \dim \, \SPS{3}{4} = 15 + 3,\quad  \dim \, \SPS{5}{6} = 28 + 3.
$$
In the following, we verify whether symmetric quadrature rules that are exact on $\mathbb{P}_{2r}$ remain exact on $\SPS{2r-1}{2r}$.

Similar to the case $\DCT$, we first construct a basis for $\SPS{2r-1}{2r}$ consisting of normalized simplex splines.
The graphical notation
\begin{equation}\label{eq:PS-graphic}
\PPS{i}{j}{k}{\ell}{m}{n}
\end{equation}
is used to represent the normalized simplex spline $ N_\mathcal{P}$ specified by the sequence of points
$$\mathcal{P}= \lbrace \p_1[i], \p_2[j], \p_3[k], \p_{1,2}[\ell], \p_{2,3}[m], \p_{3,1}[n] \rbrace.$$
Some examples are visualized in Figure~\ref{fig:PS-splines}. Note that
$$
\left\lbrace  \PPS{i}{j}{k}{0}{0}{0}, \ i,j,k\geq 1, \ i+j+k=2r+3
\right\rbrace
$$
is the set of Bernstein polynomials of degree $2r$ related to $T$.

\begin{figure}[t!]
\centering
\includegraphics[scale=0.45]{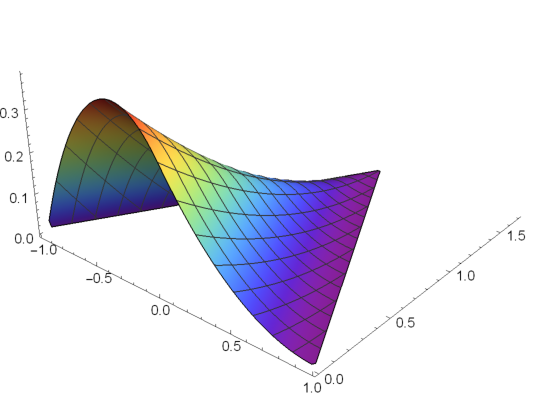} \includegraphics[scale=0.45]{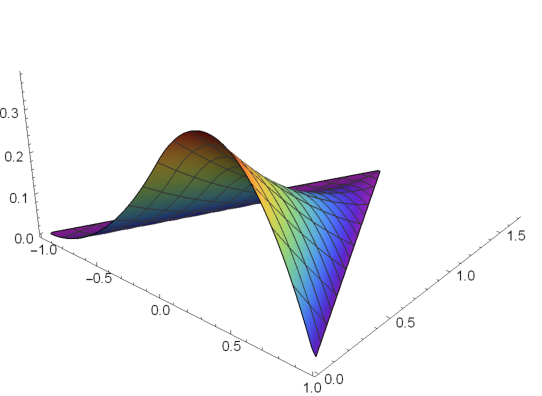} \includegraphics[scale=0.45]{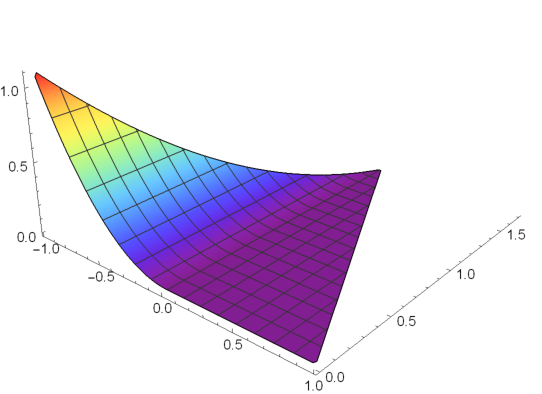}
\\[0.2cm]
$\put(-135,20){ \PPS{2}{1}{1}{1}{0}{0}} \put(-10,20){ \PPS{1}{2}{1}{1}{0}{0}} \put(110,20){ \PPS{3}{0}{1}{1}{0}{0}}$
\caption{Plots of simplex splines as in \eqref{eq:PS-graphic} of degree $2$.}\label{fig:PS-splines}
\end{figure}

As before, a basis for $\SPS{2r-1}{2r}$ can be obtained from the set of Bernstein polynomials of degree $2r$, by applying knot insertion to few of them.
More precisely, we consider the three Bernstein polynomials
$$
 \PPS{1}{r+1}{r+1}{0}{0}{0}, \quad  \PPS{r+1}{1}{r+1}{0}{0}{0}, \quad  \PPS{r+1}{r+1}{1}{0}{0}{0},
$$
and for each of them we insert the midpoint of the edge of $T$ where it does not vanish. In other words, we insert the point $\p_{2,3}$, $\p_{3,1}$, and $\p_{1,2}$, respectively.
Then, the knot insertion formula gives
\begin{equation}\label{eq:Bern-edge}
\begin{aligned}
{\PPS{1}{r+1}{r+1}{0}{0}{0}} &= \frac{1}{2} \left( \PPS{1}{r}{r+1}{0}{1}{0} + \PPS{1}{r+1}{r}{0}{1}{0}\right), \\
{\PPS{r+1}{1}{r+1}{0}{0}{0}} &= \frac{1}{2} \left( \PPS{r+1}{1}{r}{0}{0}{1} + \PPS{r}{1}{r+1}{0}{0}{1}\right), \\
{\PPS{r+1}{r+1}{1}{0}{0}{0}} &= \frac{1}{2} \left( \PPS{r}{r+1}{1}{1}{0}{0} + \PPS{r+1}{r}{1}{1}{0}{0}\right).
\end{aligned}
\end{equation}

The next lemma serves to prove the linear independence of the simplex splines we are dealing with.
\begin{lemma}\label{lem:PS-jump}
For $c_1,c_2,\ldots,c_6\in\mathbb{R}$, let
$$
p = c_1 \PPS{1}{r}{r+1}{0}{1}{0} + c_2\PPS{1}{r+1}{r}{0}{1}{0}+ c_3\PPS{r+1}{1}{r}{0}{0}{1} + c_4\PPS{r}{1}{r+1}{0}{0}{1}+ c_5\PPS{r}{r+1}{1}{1}{0}{0} + c_6\PPS{r+1}{r}{1}{1}{0}{0}.
$$
Then, $p \in \mathbb{P}_{2r}$ if and only if
$ c_1 = c_2$, $c_3=c_4$, and $c_5=c_6$.
\end{lemma}
\begin{proof}
Taking into account \eqref{eq:Bern-edge},
the same line of arguments as in the proof of Lemma~\ref{lem:CT-jump} can be used to obtain this result.
\end{proof}

We now provide a normalized simplex spline basis for $\SPS{2r-1}{2r}$.
\begin{proposition} \label{prop:PS-basis}
The normalized simplex splines belonging to the set
\begin{equation}\label{eq:PS-basis}
\begin{aligned}
\Bps{r} &:= \left\lbrace  \PPS{i}{j}{k}{0}{0}{0}, \ i,j,k\geq 1, \ i+j+k=2r+3
\right\rbrace \\
&\quad\ \setminus\left\lbrace  \PPS{1}{r+1}{r+1}{0}{0}{0}, \ \PPS{r+1}{1}{r+1}{0}{0}{0}, \PPS{r+1}{r+1}{1}{0}{0}{0}
\right\rbrace \\
&\hspace*{-1.75cm}\quad\ \bigcup
\left\lbrace  \PPS{1}{r}{r+1}{0}{1}{0},\ \PPS{1}{r+1}{r}{0}{1}{0},\ \PPS{r+1}{1}{r}{0}{0}{1},\ \PPS{r}{1}{r+1}{0}{0}{1},\ \PPS{r}{r+1}{1}{1}{0}{0},\ \PPS{r+1}{r}{1}{1}{0}{0}
\right\rbrace
\end{aligned}
\end{equation}
form a basis for $\SPS{2r-1}{2r}$.
\end{proposition}
\begin{proof}
The cardinality of $\Bps{r} $ is clearly $\dim \, \mathbb{P}_{2r}+3$ and thus agrees with the dimension of $\SPS{2r-1}{2r}$. Moreover,
from the properties of simplex splines, all the elements of $\Bps{r} $ belong to $\SPS{1}{2}$. Their linear independence follows from Lemma~\ref{lem:PS-jump}, with the same line of arguments as in the proof of Proposition~\ref{prop:CT-basis}.
\end{proof}

We are now ready to present the main result of this section.
\begin{theorem}\label{thm:PS-quad}
Any quadrature rule $\qd{[n_0, n_1, n_2]}{\mathbb{P}_{2r}}$ is exact on the spline space $\SPS{2r-1}{2r}$.
\end{theorem}
\begin{proof}
It suffices to show that the quadrature rule produces the exact value for the integral of the basis elements in \eqref{eq:PS-basis}. Since the quadrature rule is exact on ${\mathbb{P}_{2r}}$, we only consider the basis elements that are not polynomials on $T$.
Starting from the first equality in \eqref{eq:Bern-edge} and exploiting the symmetry of both the quadrature rule and the simplex splines, we deduce
$$
\begin{aligned}
\qd{[n_0, n_1, n_2]}{\mathbb{P}_{2r}} \left( \PPS{1}{r+1}{r+1}{0}{0}{0} \right)  &=  \qd{[n_0, n_1, n_2]}{\mathbb{P}_{2r}} \left( \PPS{1}{r}{r+1}{0}{1}{0} \right) \\
&= \qd{[n_0, n_1, n_2]}{\mathbb{P}_{2r}} \left( \PPS{1}{r+1}{r}{0}{1}{0} \right).
\end{aligned}
$$
In addition, again from the first equality in \eqref{eq:Bern-edge} and the symmetry of the simplex splines, we get
$$
\int_T \PPS{1}{r+1}{r+1}{0}{0}{0} = \int_T \PPS{1}{r}{r+1}{0}{1}{0} = \int_T \PPS{1}{r+1}{r}{0}{1}{0}.
$$
Since $\qd{[n_0, n_1, n_2]}{\mathbb{P}_{2r}}$ is exact for $\PPS{1}{r+1}{r+1}{0}{0}{0}$, it follows that it is also exact for $\PPS{1}{r}{r+1}{0}{1}{0}$ and $\PPS{1}{r+1}{r}{0}{1}{0}$.
By means of similar arguments applied to the second and third equality in \eqref{eq:Bern-edge}, we conclude the proof.
\end{proof}

\begin{example}
The Hammer--Stroud quadrature rule
\begin{equation}\label{eq:HS-quadratic}
\mathcal{Q}_{\mathbb{P}_2}^{\mathrm{HS}} (f) :=  \frac{\area (T)}{3} \sum_{ \bm{\alpha} \in \Pi \left\lbrace \frac{2}{3}, \frac{1}{6}, \frac{1}{6} \right\rbrace } f \left( \T \left(\bm{\alpha} \right)\right)
\end{equation}
is exact for quadratic polynomials \cite{HM56}. The rule $\mathcal{Q}_{\mathbb{P}_2}^{\mathrm{HS}}$ is a quadrature rule of type $\qd{[0, 1, 0]}{\mathbb{P}_{2}}$. From Theorem~\ref{thm:PS-quad} (with $r=1$) we know that it exactly integrates any element of the space $\SPS{1}{2}$. The same result has already been obtained in \cite{BK19} by using a different approach.
\end{example}

\begin{remark}
The smoothness $C^{2r-1}$ is the minimum for splines of degree $2r$ on $\DPS$ to ensure that any quadrature rule of type $\qd{[n_0, n_1, n_2]}{\mathbb{P}_{2r}}$ remains exact on the spline space. This can be seen by considering the simplest case $r=1$ and the element of $\SPS{0}{2}$ given by
$$
\PPS{2}{0}{1}{2}{0}{0} = 2 \left( \alpha_1-\alpha_2 \right) \left(2 \alpha_2\right) \PPS{1}{0}{1}{1}{0}{0},
$$
which is $C^0$ across the edge $\langle \p_3, \p_{1,2} \rangle $.
We take the Hammer--Stroud rule $\mathcal{Q}_{\mathbb{P}_2}^{\mathrm{HS}}$ defined in \eqref{eq:HS-quadratic} and a direct computation shows that
 $$\mathcal{Q}_{\mathbb{P}_2}^{\mathrm{HS}}\left ( \PPS{2}{0}{1}{2}{0}{0}\right )= \frac{2\area (T)}{9},\quad
 \text{ while} \quad
 \int_T \PPS{2}{0}{1}{2}{0}{0}=\frac{\area (T)}{6}.$$
Therefore, the rule $\mathcal{Q}_{\mathbb{P}_2}^{\mathrm{HS}}$ fails to achieve exactness on the space $\SPS{0}{2}$.
\end{remark}

\begin{remark}
Analogous to the even degree case, by \eqref{eq:PS-dim-full} we have
$\dim \, \SPS{2r}{2r+1} = \dim \, \mathbb{P}_{2r+1} + 3$.
However, Theorem~\ref{thm:PS-quad} does not possess a similar counterpart for odd degrees. Indeed, the Hammer--Stroud rule $\mathcal{Q}_{\mathbb{P}_3}^{\mathrm{HS}}$ defined in \eqref{eq:HS-cubic} is not able to exactly integrate the space $\SPS{2}{3}$ (and consequently, it is not able to exactly integrate the spaces $\SPS{\rho}{3}$ for $\rho \leq 2$). This can be seen by considering the following spline in $\SPS{2}{3}$:
$$
\PPS{4}{0}{1}{1}{0}{0} = \left( \alpha_1-\alpha_2\right)^3 \PPS{1}{0}{1}{1}{0}{0}.
$$
A direct computation shows that
$$\mathcal{Q}_{\mathbb{P}_3}^{\mathrm{HS}}\left (\PPS{4}{0}{1}{1}{0}{0} \right) = \frac{\area (T)}{15}, \quad
\text{ while}\quad
\int_T \PPS{4}{0}{1}{1}{0}{0} = \frac{\area (T)}{10}.$$
Therefore, the rule $\mathcal{Q}_{\mathbb{P}_3}^{\mathrm{HS}}$ fails to achieve exactness on the space $\SPS{2}{3}$.
\end{remark}

\section{Conclusion}\label{sec:conclusion}
Symmetric quadrature rules on triangles that are exact for polynomials of a specific degree $d$ can maintain their exactness when applied to larger spline spaces of the same degree $d$ in certain uniform scenarios. In this study, we have investigated such preservation of exactness for non-trivial spline spaces of the highest smoothness, having degree $d=3r$ on the Clough--Tocher 3-split or $d=2r$ on the (uniform) Powell--Sabin 6-split, for any positive integer~$r$.

Both the splits addressed in the current work have natural counterparts in $\mathbb{R}^m$ ($m\geq3$). The Alfeld split \cite{A84} extends the Clough--Tocher 3-split while the Worsey--Piper split \cite{WP88} can be seen as a generalization of the Powell--Sabin 6-split for $m=3$. Furthermore, the main tool used to prove the results of this paper, i.e., the simplex spline theory, is also perfectly suited for the $m$-variate setting.
In particular, for the Alfeld split in $\mathbb{R}^3$, the minimum degree required to achieve $C^1$ continuity is four, and the dimension of trivariate $C^1$ quartic splines on the Alfeld split is $38$, i.e., the dimension of trivariate quartic polynomials plus three \cite{KS13}. This suggests to replace one Bernstein polynomial by four simplex splines to get a basis of trivariate $C^1$ quartic splines on the Alfeld split, in a complete analogy with the bivariate case \cite{LMS24}. The line of arguments used in Section~\ref{sec:CT-split} can be extended to prove that a symmetric quadrature rule that exactly integrates trivariate quartic polynomials remains exact for trivariate $C^1$ quartic splines on the Alfeld split.
This indicates that the methodology employed in this work can be extended to the $m$-variate setting. This extension will be addressed in future work.

\bibliographystyle{amsplain}

\end{document}